\documentclass[11pt,leqno]{amsart}
\usepackage{amsmath,amsfonts,latexsym,graphicx,amssymb,url, color}
\usepackage{txfonts} 

\usepackage[hmargin=3.35 cm,vmargin=3.35 cm]{geometry}

\usepackage{color} 

\newcommand{\R}{{\mathbb R}} 

\newcommand{\A}{\mathbf{A}}

\newcommand\norm[1]{\left\| #1\right\|}

\newcommand{\wei}[1]{\langle #1 \rangle}
\newcommand{\sgn}{\text{sgn}}

\newcommand{\F}{{\mathbf F}}

\newtheorem{theorem}{Theorem}[section]

\newtheorem{lemma}[theorem]{Lemma}

\numberwithin{equation}{section}

\newcommand{\beq}{\begin{equation}}
\newcommand{\eeq}{\end{equation}}

\definecolor{darkred}{rgb}{.70,.12,.20}

\definecolor{darkgreen}{rgb}{.20,.52,.14}

\title[Higher integrability estimates, elliptic equations with singular coefficients] {On higher integrability estimates for elliptic equations with singular coefficients}
\author{Juraj F\"{o}ldes and Tuoc Phan}
\address[J. F\"{o}ldes]{Department of Mathematics, University of Virginia, 141 Cabell Drive, Kerchof Hall
P.O. Box 400137, Charlottesville, VA 22904, U.S.A.}
\email{foldes@virginia.edu}
\address[T. Phan]{Department of Mathematics, University of Tennessee, Knoxville, 227 Ayres Hall, 1403 Circle Drive, Knoxville, TN 37996, U.S.A.}
\email{phan@utk.edu}

\begin{document}

\maketitle
\begin{abstract}  In this note 
we establish existence and uniqueness of weak solutions of linear elliptic equation $\text{div}[\A(x) \nabla u] = \text{div}{\F(x)}$, where the matrix $\A$ is just measurable and its skew-symmetric part can be unbounded. 
Global reverse H\"{o}lder's regularity estimates for gradients of weak solutions are also obtained. Most importantly, we show, by providing an example, that boundedness  and ellipticity of $\A$ is not sufficient for higher integrability estimates even when the symmetric part of $\A$ is the identity matrix.  In addition, 
the example also shows the necessity of the dependence of $\alpha$ in the H\"{o}lder $C^\alpha$-regularity theory on the \textup{BMO}-semi norm of the skew-symmetric part of $\A$. 
The paper is an extension of classical results obtained by N. G. Meyers (1963) in which the skew-symmetric part of $\A$ is assumed to be zero.
\end{abstract}

\section{Introduction}
Let $\Omega \subset \mathbb{R}^n$ be an open bounded  domain with Lipschitz boundary $\partial \Omega$. Let $\A :\Omega \rightarrow \mathbb{R}^{n\times n}$ be a given measurable matrix, and $\F: \Omega \rightarrow \R^n$ be a given measurable vector field. In this note, we study the linear elliptic problem
\begin{equation} \label{main-eqn}
\left\{
\begin{aligned}
\textup{div}[\A(x) \nabla u(x)] &=  \textup{div }  \F(x), && \quad \text{in} \quad \Omega, \\
u &=  0, && \quad \text{on} \quad \partial \Omega\,,
\end{aligned}
\right.
\end{equation}
where  the coefficient matrix $\A(x)$ could be non-symmetric and singular. Specifically, we write $\A(x) = a(x) + d(x)$, where $a = (a_{ij})_{n\times n}$ is a symmetric and $d = (d_{ij})_{n\times n}$ a skew-symmetric part of $\A$, that is, 
$a_{ij} = a_{ji}$ and $d_{ij}= - d_{ji}$ for all $i, j \in\{1,2,\cdots, n\}$.  We assume that $a$ is uniformly elliptic and bounded, meaning that  there exists a constant $\Lambda >0$ such that
\begin{equation} \label{ellipticity-cond}
\Lambda |\xi|^2 \leq \wei{a(x) \xi, \xi}, \quad |a(x)|  \leq \Lambda^{-1}, \quad \forall \ \xi \in \R^n, \quad \text{for a.e.} \quad  x \in \Omega.
\end{equation}
However,  the skew-symmetric part $d : \R^n \rightarrow \R^{n\times n}$ is only in the Jonh-Nirenberg \text{BMO}-space, and therefore $d$ can be unbounded. Precisely, we assume that
\begin{equation} \label{BMO-d}
[[d]]_{\text{BMO}} = \sup\left\{\fint_{B_\rho(x)} | d(y) - \overline{d}_{B_\rho(x)}| dy, \quad 
x \in \mathbb{R}^n, \quad  \rho >0 \right\} < \infty,
\end{equation}
where 
\begin{equation} \label{dav}
\overline{d}_{B_\rho(x)} = \fint_{B_\rho(x) } d(y) dy
\end{equation}
is the average of $d$ over the ball $B_\rho(x)$.
Observe that the differential equations in \eqref{main-eqn} can be formally written as an elliptic equation with drift (or advection term) 
\begin{equation} \label{drift}
\text{div}[a(x) \nabla u] + b \cdot \nabla u = \text{div}[\F(x)],
\end{equation}
where the vector field $b$ is defined by
\[
b_k(x) = \sum_{l=1}^n \frac{\partial d_{kl}(x)}{\partial x_l}.
\]
Since $(d_{kl})$ is skew-symmetric, the vector field $b$ is divergence-free. Due to the interests from many problems in fluid mechanics, biology, and probability, the class of equation \eqref{drift} has attracted great attention, see
for example \cite{Dong-Kim, KS, LP, LS, O, TP, TP1,  Sverak, Q-Zhang}.

The first result of this paper extends the classical results established in \cite{M} to the equations \eqref{main-eqn} and \eqref{drift}. 

\begin{theorem} \label{self-improvement} Let $\Lambda >0$ and assume that $\A$ satisfies \eqref{ellipticity-cond}-\eqref{BMO-d}. 
 Then, for every $\F \in L^2(\Omega)$, there exists a unique weak solution $u \in W^{1,2}_0(\Omega)$ of \eqref{main-eqn}, that satisfies 
\begin{equation} \label{energy-thm}
\int_{\Omega} |\nabla u(x)|^2 dx \leq C(\Lambda, [[d]]_{\textup{BMO}}) \int_{\Omega} |\F(x)|^2 dx.
\end{equation}
Moreover, there exists $\epsilon_0 > 0$ depending only on $\Lambda, n, \Omega$, and $[[d]]_{\textup{BMO}}$ such that if $u$ is a weak solution of \eqref{main-eqn} and $\F \in L^p(\Omega)$ for some $ p \in [2, 2 + \epsilon_0]$, then
\begin{equation} \label{reverse-Holder}
\norm{\nabla u}_{L^p(\Omega)} \leq C(\Lambda, n, \Omega, p, [[d]]_{\textup{BMO}}) \norm{\F}_{L^p(\Omega)}.
\end{equation}
\end{theorem}

Our result is new, since coefficients of $\A$ can be unbounded as the skew-symmetric part $d$ is only assumed to be in the John-Nirenberg $\textup{BMO}$ space. 
Note that the type of estimate \eqref{reverse-Holder} is usually called reverse H\"{o}lder's estimate and is very important in may contexts, see e.g. \cite{Giaquinta, Gia-Modica, Gia-Stru, Giusti, Sverak}. 
The bound \eqref{reverse-Holder} was first established  in \cite{M} for elliptic and bounded matrix $\A$, i.e. $d$ is also bounded. Our result is 
a natural extension of \cite{M}, which covers an important class of equation \eqref{drift}. 
Note also that  in \cite{Sverak} were proved interior estimates  analogous to \eqref{reverse-Holder} for a  parabolic version of \eqref{drift}. See also \cite{Dong-Kim} and \cite{LP} for other similar and related results regarding the equations \eqref{main-eqn} and \eqref{drift}. In comparison our result also yields bounds up to the boundary and with nontrivial right hand side. 

Our next theorem is the main contribution of this paper. This theorem shows that the dependence of $\epsilon_0$ defined in Theorem \ref{self-improvement} on $[[d]]_{\textup{BMO}}$ is in general necessary.  Moreover, this theorem also shows the nonlinear dependence of $\alpha$ on $[[d]]_{\textup{BMO}}$ in the $C^\alpha$-regularity estimates established in \cite{Sverak}. 
\begin{theorem}\label{exam-thrm} Let $\mu \in (0,1)$, there exists a measurable function $d: \mathbb{R}^2 \rightarrow \mathbb{R}$ with $\norm{d}_{L^\infty(\mathbb{R}^2)} = \frac{\pi (1-\mu^2)}{2\mu}$, and there is a weak solution $u \in W^{1,2}(B_1)$ of
\[
\textup{div}[\A(x) \nabla u] = 0, \quad \text{in} \quad B_1
\]
such that $\nabla u \not \in L^{p}(B_{1/2})$ for $p \geq \frac{2}{1-\mu}$ and $u \not\in C^\alpha(B_{1/2})$ for $\alpha \geq \mu$, where
\[
\A(x) = \left(\begin{matrix} 1 & 0 \\ 0 & 1 \end{matrix} \right) + \left(\begin{matrix} 0 & d(x) \\ -d(x)  & 0 \end{matrix} \right), \quad x \in \mathbb{R}^2,
\]
and $B_\rho \subset \R^2 $ denotes the ball of radius $\rho>0$ centered at the origin.
\end{theorem}

We emphasize again that  Theorem \ref{exam-thrm} also shows the nonlinear dependence of $\epsilon_0$ defined in Theorem \ref{self-improvement} on the $[[d]]_{\textup{BMO}}$.  In particular, Theorem \ref{exam-thrm} shows that the regularity estimate \eqref{reverse-Holder} fails to hold for arbitrary large $p > 2$, even when the symmetric part of $\A$ is the identity matrix and the skew-symmetric part is bounded.  In addition, Theorem \ref{exam-thrm} also shows that local $C^\alpha$-norm of the solution is not bounded for any $\alpha > 0$ if $[[d]]_{\textup{BMO}}$ is sufficiently large, showing that the dependence of $\alpha$ on $[[d]]_{\textup{BMO}}$ cannot be in general 
removed in  \cite[Lemma 2.9]{Sverak}. Observe that an analogous example is also provided in 
\cite{M}, for non-smooth bounded symmetric part of $\A$ and trival skew-symmetric part.  In comparison our example deals with problems with smooth diffusion and singular, divergence free drifts. The two examples demonstrate that both $[[d]]_{\textup{BMO}}$ and $[[a]]_{\textup{BMO}}$ of $[[\A]]_{\textup{BMO}}$ are have to be sufficiently small for establishing Calder\'{o}n-Zydmund type estimates for weak solutions of \eqref{main-eqn} and \eqref{drift}, see \cite{TP, TP1}.

In the rest of the paper we prove Theorems \ref{self-improvement}-\ref{exam-thrm}.  The proof of theorem \ref{self-improvement} is given in Section \ref{proof-self-thrm}, and the proof of Theorem \ref{exam-thrm} is provided in Section \ref{example-section}.

\section{Proof of Theorem \ref{self-improvement}} \label{proof-self-thrm}

The proof of Theorem \ref{self-improvement} is quite standard and it could be known by experts. We provide it here for completeness. We need several lemmas. We first recall two classical analysis lemmas that are needed in the proof. The first lemma is a Poincar\'{e}-Sobolev inequality, whose proof can be found in \cite[p. 13]{Ben-Fre}.

\begin{lemma}[Poincar\'e-Sobolev inequality] \label{Poincare} Let $\Omega \subset \R^n$ be a Lipschitz domain,  $\mu \in [1, 2]$, and let $\lambda >0$ be such that
\[
n \left( \frac{1}{\lambda} - \frac{1}{\mu} \right) + 1 \geq 0.
\]
If $u \in W^{1,2}_0(\Omega)$, then for every $r >0, x_0 \in \overline{\Omega}$ 
\[
\left(\int_{\Omega_{r}(x_0)} |u(x) - (u)_{B_r(x_0)}|^\lambda \right)^{1/\lambda} 
\leq C(n) r^{ n \big(\frac{1}{\lambda} - \frac{1}{\mu}  \big) +1} \left( \int_{\Omega_{4r/3}(x_0)} |\nabla u|^\mu dx \right)^{1/\mu} \,,
\]
where
\begin{equation} \label{average}
(u)_{B_r(x)} = 
\left\{
\begin{array}{ll}
\fint_{B_r(x)} u(y) dy, & \quad \text{if} \quad B_r(x) \subset \Omega, \\
0 & \quad \text{if} \quad B_r(x) \cap (\mathbb{R}^n \setminus \Omega) \not=\emptyset.
\end{array}
\right.
\end{equation}
\end{lemma}

\noindent
In the next lemma, we denote $Q_r(x_0)$  a cube in $\mathbb{R}^n$ with edge of length $2r >0$, centered $x_0 \in \mathbb{R}^n$. The following Gehring-type estimate is due to M. Giaquinta and G. Modica \cite{Gia-Modica} and the following
statement is proved in \cite[Theorem 1.10, p. 25]{Ben-Fre}, see also \cite[Proposition 1.1, p. 122]{Giaquinta}.

\begin{lemma}[Gehring type estimate] \label{Gehring} Fix $Q$ a bounded cube in $\mathbb{R}^n$, $ 1 < q <  l$, and let $f, g$ be non-negative functions  such that $g \in L^q(Q)$ and $f \in L^l(Q)$. Assume that there exists $b >1$ such that
\[
\fint_{Q_r(x_0)} g^q dx \leq b \left[ \left(\fint_{Q_{2r}(x_0)} g dx \right)^{q} + \fint_{Q_{2r}(x_0)} f^q dx \right]
\]
for each $Q_r(x_0)$ so that $Q_{2r}(x_0) \subset Q$. Then, there exists $\epsilon = \epsilon(q, n, b) >0$ such that for every $p \in [q, q + \min\{l-q,\epsilon\})$, there is $C = C(q,n, b, p)$ such that
\[
\left( \fint_{Q_r} g^p dx \right)^{1/p} \leq C \left[ \left(\fint_{Q_{2r}} g^q dx \right)^{1/q} + \left( \fint_{Q_{2r}} f^p dx \right)^{1/p} \right],
\]
for each cube $Q_r$ satisfying $Q_{2r} \subset Q$.
\end{lemma}

Next, for each $x \in \overline{\Omega}$, and each $r >0$, we write $B_r(x)$ the open ball in $\mathbb{R}^n$ with radius $r$ and centered at $x$. Moreover, we write
\[
\Omega_r(x) = B_r(x) \cap \Omega.
\]
In the next lemma, we establish Caccioppoli-type estimates for weak solutions of \eqref{main-eqn}.

\begin{lemma}[Caccioppoli type inequality] \label{Caccioppoli} 
Fix $s \in (1,2)$, $\Lambda >0$, $\Omega \subset \R^n$ and assume that the $n \times n$ matrix $\A$ satisfies \eqref{ellipticity-cond}-\eqref{BMO-d}. 
Then,  there is a constant  $C$ depending only on $\Lambda$, $s$, and $n$ such that for any weak solution $u$ of \eqref{main-eqn}, any $r >0$ and $x_0 \in \overline{\Omega}$, one has
\[
\int_{\Omega_{r}(x_0)} |\nabla \hat{u} |^2 dx \leq C(\Lambda, n ,s) \left[ r^{\frac{2n}{s'} -2}\Big[ [[d]]^2_{\textup{BMO}} +1 \Big] \left\{ \int_{\Omega_{3r/2}(x_0)} |\hat{u}|^{\frac{2s}{2-s}} dx \right\}^{\frac{2-s}{s}} +\int_{\Omega_{3r/2}(x_0)} |\F(x)|^2 dx \right ].
\]
where  $s' = \frac{s}{s-1}$ and $\hat{u} = u - (u)_{B_r(x_0)}$ with $(u)_{B_r(x_0)}$ is defined in \eqref{average}.
\end{lemma}

\begin{proof} 
Without loss of generality, we assume that $x_0 =0$. We also write $\Omega_r = \Omega_r(0)$ and $B_r = B_r(0)$. Let $\varphi \in C_0^\infty(B_{3r/2}, [0,1])$ be a cut-off function with 
$\varphi =1$ on $B_r$ and $|\nabla \varphi| \leq \frac{2}{r}$. Note that if $B_{3r/2} \cap (\mathbb{R}^n \setminus \Omega) \not= \emptyset$, then $(u)_{B_{3r/2}} = 0$, and therefore
\[
\hat{u} = u = 0 \quad \text{on} \quad B_{3r/2} \cap \partial \Omega.
\]
By using $\hat{u}\varphi^2$ as a test function, we obtain
\[
\begin{split}
\int_{\Omega_{3r/2}} \wei{a \nabla\hat{u}, \nabla \hat{u} }\varphi^2 dx & = 
-2\int_{\Omega_{3r/2}} \wei{a \nabla \hat{u}, \nabla \varphi} \varphi \hat{u} dx - \int_{\Omega_{3r/2}} \wei{d \nabla \hat{u}, \nabla \hat{u} }\varphi^2 dx \\
&  \quad -  \int_{\Omega_{3r/2}} \wei{d \nabla \hat{u}, \nabla (\varphi^2)} \hat{u} dx 
+ \int_{\Omega_{3r/2}} \F(x) \cdot [ \varphi^2 \nabla \hat{u} + 2 \varphi \hat{u} \nabla \varphi] dx.
\end{split}
\]
Since $d = - d^*$, we have
\[
\wei{d\nabla \hat{u}, \nabla \hat{u}} =0, \quad \int_{\Omega_{3r/2}} \wei{\overline{d}_{B_{2r}(0)} \nabla \hat{u}, \nabla (\varphi^2)} \hat{u} dx =0 \,,
\]
where $\overline{d}_{B_{2r}(0)}$ is the average of $d$ defined in \eqref{dav}. Note that the first equality is direct consequence of skew-symmetry and the other follows after two integrations by parts. 
From this, \eqref{ellipticity-cond}, and Young's inequality, we obtain
\[
\begin{split}
\Lambda \int_{\Omega_{3r/2}} |\nabla \hat{u}|^2\varphi^2 dx & \leq 
2\Lambda^{-1}\int_{\Omega_{3r/2}} |\nabla \hat{u}| |\nabla \varphi| \varphi |\hat{u}| dx 
 + 2 \int_{\Omega_{3r/2}} |d - \bar{d}_{B_{3r/2}}| | \nabla \hat{u}|  |\nabla \varphi| \varphi |\hat{u}| dx \\
 & \quad \quad
+ \int_{\Omega_{3r/2}}|\F|[ \varphi^2 |\nabla \hat{u}| + 2|\nabla \varphi| |\hat{u}|\varphi] dx  \\
& \leq  \frac{\Lambda}{2} \int_{\Omega_{3r/2}} |\nabla \hat{u}|^2 \varphi^2 dx 
+ C(\Lambda) \int_{\Omega_{3r/2}} \Big[ |\F|^2\varphi^2 + |\nabla \varphi|^2 \hat{u}^2 \Big ]dx   \\
& \quad + C(\Lambda) \int_{\Omega_{3r/2}} |d - \bar{d}_{B_{3r/2}}|^2  |\nabla \varphi|^2  |\hat{u}|^2 dx.\end{split}
\]
Therefore,
\begin{equation} \label{energy-T}
\int_{\Omega_{3r/2}} |\nabla \hat{u}|^2\varphi^2 dx \leq C(\Lambda) \left[\frac{1}{r^2}\int_{\Omega_{3r/2}}\Big[ |d - \bar{d}_{B_{3r/2}}|^2 +1\Big] |\hat{u}|^2 dx +\int_{\Omega_{3r/2}} |\F|^2 dx   \right].
\end{equation}
We now control the first term on the right hand side of \eqref{energy-T}. Recall that $\Omega$ is a Lipschitz domain, there exists a constant $A$ such that
\[
|\Omega_\rho(x)| \geq A |B_{\rho}(x)|, \quad \forall \rho >0, \quad \forall \ x \ \in \overline{\Omega}.
\]
Then, for $s \in (1,2)$, and with $s' = s/(s-1)$, H\"{o}lder's inequality yields
\[
\begin{split}
\int_{\Omega_{3r/2}} |d - \bar{d}_{B_{2r}}|^2 |\hat{u}|^2 dx  & \leq C(n) r^{\frac{2n}{s'}}
\left\{ \fint_{B_{3r/2}} |d(x) -\bar{d}_{B_{2r}}|^{s'} dx \right\}^{\frac{2}{s'}} \left\{ \int_{\Omega_{3r/2}} |\hat{u}|^{\frac{2s}{2-s}} dx \right\}^{\frac{2-s}{s}} \\
& \leq C(n) r^{\frac{2n}{s'}} [[d]]^2_{\text{BMO}} \left\{ \int_{\Omega_{3r/2}} |\hat{u}|^{\frac{s}{2-s}} dx \right\}^{\frac{2-s}{s}}.
\end{split}
\]
Moreover, since
\[
\int_{\Omega_{3r/2}} |\hat{u}|^2 dx \leq C(n) r^{\frac{2n}{s'}}\left\{ \int_{\Omega_{3r/2}} |\hat{u}|^{\frac{2s}{2-s}} dx \right\}^{\frac{2-s}{s}},
\]
it follows from the last two estimates and \eqref{energy-T} that
\[
\int_{\Omega_{3r/2}} |\nabla u|^2\varphi^2 dx \leq C(\Lambda,n ,s) \left[ r^{\frac{2n}{s'} -2}\Big[ [[d]]^2_{\text{BMO}} +1 \Big] \left\{ \int_{\Omega_{3r/2}} |\hat{u}|^{\frac{2s}{2-s}} dx \right\}^{\frac{2-s}{s}} +\int_{\Omega_{3r/2}} |\F|^2\varphi^2 dx   \right] \,,
\]
and the proof is complete.
\end{proof}

\begin{lemma}[Reverse H\"{o}lder's inequality] \label{RH} For $ \mu \in \Big (\frac{2n}{n+2},2 \Big)$, there exists a constant $C = C(\Lambda, n, \mu)$ such that for every $x_0 \in \overline{\Omega}$ and every $r >0$, we have 
\[
 \int_{\Omega_{r}(x_0)} |\nabla \hat{u}|^2 dx \leq C 
\left\{r^{n\big(1 - \frac{2}{\mu}\big)} \Big[ [[d]]^2_{\textup{BMO}} +1 \Big] \left ( \int_{\Omega_{2r}(x_0)} |\nabla \hat{u}|^{\mu} \right )^{2/\mu} + \int_{\Omega_{2r}(x_0)} |\F(x)|^2 dx  \right\}.
\]
\end{lemma}

\begin{proof}
 Since $\mu \in \Big (\frac{2n}{n+2},2 \Big)$, we can choose $s \in (1,2)$ such that  $\lambda = \frac{2s}{2-s}$ satisfies
\[
n \left( \frac{1}{\lambda} - \frac{1}{\mu} \right)+1 \geq 0.
\]
 Then, it follows from Poincar\'e - Sobolev's inequality, Lemma \ref{Poincare}, that
\[
\left( \int_{\Omega_{3r/2}} |\hat{u}|^{\lambda} dx \right)^{1/\lambda} \leq C(n, \mu) r^{n \Big(\frac{1}{\lambda} - \frac{1}{\mu} \Big) +1} 
\left (\int_{\Omega_{2r}} |\nabla \hat{u}|^{\mu} \right )^{1/\mu}.
\]
Therefore, Lemma \ref{Caccioppoli} implies
\[
\begin{split}
\int_{\Omega_{r}} |\nabla \hat{u}|^2 dx  & \leq C(\Lambda,n ,s) \left\{ r^{2n\big( \frac{1}{s'} + \frac{1}{\lambda} - \frac{1}{\mu}\big)}\Big[ [[d]]^2_{\text{BMO}} +1 \Big] \left (\int_{\Omega_{2r}} |\nabla \hat{u}|^{\mu} \right )^{2/\mu} +\int_{\Omega_{2r}} |\F(x)|^2 dx   \right\} \\
&= C(\Lambda,n ,s) \left\{ r^{n\big(1 - \frac{2}{\mu}\big)}\Big[ [[d]]^2_{\text{BMO}} +1 \Big] \left (\int_{\Omega_{2r}} |\nabla \hat{u}|^{\mu} \right )^{2/\mu} +\int_{\Omega_{2r}} |\F(x)|^2 dx   \right\}.
\end{split}
\]
The proof is then complete.
\end{proof}
\begin{proof}[Proof of Theorem \ref{self-improvement}] We start with proving the existence and uniqueness of the weak solution of \eqref{main-eqn}. To this end, we define the bilinear form
\[
B(u, v) = \int_{\Omega}\wei{\A(x) \nabla u, \nabla v} dx, \quad u, v \in W^{1,2}_0(\Omega),
\]
and prove that it is coercive and bounded. First of all, note that by the ellipticity condition \eqref{ellipticity-cond} and the fact that $d$ is skew-symmetric, we see that
\[
B(u,u) \geq \Lambda \norm{\nabla u}_{L^2(\Omega)}^2, \quad \forall \ u \in W^{1,2}_0(\Omega).
\]
Hence, $B$ is coercive. On the other hand, since $d= (d_{ij})_{n \times n}$ skew-symmetric, and it is in $\textup{BMO}$, it follows from \cite{CLMS} and \cite{FS} (see also \cite{MV}) that
\[
\left| \int_{\Omega} \wei{d \nabla u, \nabla v} dx \right| \leq C(n, [[d]]_{\textup{BMO}}) \norm{\nabla u}_{L^2(\Omega)} \norm{\nabla u}_{L^2(\Omega)},  \quad \forall \ u, v \in W^{1,2}_0(\Omega).
\]
This together with \eqref{ellipticity-cond} imply that
\[
|B(u,v)| \leq C(n, \Lambda, [[d]]_{\textup{BMO}}) \norm{\nabla u}_{L^2(\Omega)} \norm{\nabla u}_{L^2(\Omega)},  \quad \forall \ u, v \in W^{1,2}_0(\Omega).
\]
This gives the boundedness of $B$. Therefore, the existence, uniqueness of weak solution of \eqref{main-eqn}, and the estimate \eqref{energy-thm}  follow from the Lax-Milgram theorem.

It now remains to prove \eqref{reverse-Holder}. Let us cover $\Omega$ with 
a finite number of balls
\[
\Omega \subset \left[ \bigcup_{k=1}^{m_0} B_{\rho_k}(x_k) \right] \cup \left[ \bigcup_{k=1}^{l_0} B_{\mu_k}(y_k) \right], 
\]
where $\rho_k>0$, $x_k \in \partial \Omega$ for all $k =1, 2, \cdots m_0$ and 
\[
\mu_l >0, \quad B_{3\mu_l}(y_l) \subset \Omega, \quad \forall \ l = 1, 2, \cdots, l_0.
\]
Now, for each $k = 1, 2, \cdots, l_0$, it follows from Lemma \ref{RH} that
\begin{equation} \label{inter-improve}
\fint_{B_{r}} |\nabla \hat{u}|^2 dx \leq C 
\left\{\Big[ [[d]]^2_{\textup{BMO}} +1 \Big] \left ( \fint_{B_{2r}} |\nabla \hat{u}|^{\mu} \right )^{2/\mu} + \fint_{B_{2r}} |\F(x)|^2 dx  \right\}, 
\end{equation}
for all $B_{r} = B_r(x) \subset B_{\mu_k}(y_k)$.  Note that for any $R >0$ and any $z_0 \in \mathbb{R}^n$, the following inclusions are obvious
\[
B_R(z_0) \subset Q_{R}(z_0) \subset B_{R\sqrt{n}}(z_0).
\]
Therefore, we can rewrite \eqref{inter-improve} with cubes on the left- and right-hand sides with fixed ration of lengths of edges. Thus, we can apply Lemma \ref{Gehring} to find $\epsilon_k >0 $ such that with $p \in [2, 2 + \epsilon_k]$,
\begin{equation} \label{interior-im}
\left( \fint_{B_{\rho_k}(y_k)} |\nabla \hat{u}|^p dx  \right)^{1/p} \leq C 
\left\{ \left ( \fint_{B_{2\rho_k}(y_k)} |\nabla \hat{u}|^{2} \right )^{1/2} + \left(\fint_{B_{2\rho_k}(y_k)} |\F(x)|^p dx \right)^{1/p}  \right\}.
\end{equation}
Now, for $k  = 1, 2, \cdots, m_0$, we consider the ball $B_{\rho_k}(x_k)$. Note that since $x_k \in \partial \Omega$ and $\Omega$ is Lipschitz, we have
\begin{equation} \label{Lip-boundary}
\Big | B_{r}(x_k) \cap (\mathbb{R}^n \setminus \Omega) \Big| \geq c r^n.
\end{equation}
Note also $(u)_{B_{3r/2}(x_k)} = 0$, and therefore $u = \hat{u}$. By extending $\F, \hat{u}$ to be zero outside of $\Omega$, we can see that $u \in H^1_0(\mathbb{R}^n)$. Moreover, from Lemma \ref{RH} and \eqref{Lip-boundary}, we also obtain
\[
\fint_{B_{r}} |\nabla \hat{u}|^2 dx \leq C 
\left\{\Big[ [[d]]^2_{\textup{BMO}} +1 \Big] \left ( \fint_{B_{2r}} |\nabla \hat{u}|^{\mu} \right )^{2/\mu} + \fint_{B_{2r}} |\F(x)|^2 dx  \right\}, 
\]
for all $B_{r} = B_r(x) \subset B_{\rho_k}(x_k)$. As before, we can apply the Lemma \ref{Gehring} again. Hence, there exists $\delta_k >0$ such that for all $p \in [2, 2 + \delta_k]$, the following estimate holds 
\[
\left( \fint_{B_{\mu_k}(x_k)} |\nabla \hat{u}|^p dx  \right)^{1/p} \leq C 
\left\{ \left ( \fint_{B_{2\mu_k}(x_k)} |\nabla \hat{u}|^{2} \right )^{1/2} + \left(\fint_{B_{2\mu_k}(x_k)} |\F(x)|^p dx \right)^{1/p}  \right\}.
\]
This implies
\begin{equation} \label{boundary-im}
\left( \fint_{\Omega_{\rho_k}(x_k)} |\nabla \hat{u}|^p dx  \right)^{1/p} \leq C 
\left\{ \left ( \fint_{\Omega_{2\rho_k}(x_k)} |\nabla \hat{u}|^{2} \right )^{1/2} + \left(\fint_{\Omega_{2\rho_k}(x_k)} |\F(x)|^p dx \right)^{1/p}  \right\}.
\end{equation}
By taking $\epsilon = \min \{\delta_k, \epsilon_l, k = 1,2,\cdots, m_0, l = 1, 2,\cdots, l_0\}$, we see that Theorem \ref{self-improvement} follows from the estimates \eqref{energy-thm}, \eqref{interior-im}, and \eqref{boundary-im}. The proof is therefore complete.
\end{proof}
\section{Proof of Theorem \ref{exam-thrm}} \label{example-section}

This section provides an example showing that ellipticity condition and boundedness of $\A$ is not sufficient for \eqref{reverse-Holder} for large $p > 2$, even when the symmetric part of $\A$ is the identity matrix.
Our construction is partly motivated by a famous work \cite{M}, in which the symmetric part of $\A$ is bounded (but not continuous)
 and  the skew-symmetric part is identically zero.  In our example, the symmetric part of $\A$ is the identity matrix, but  the skew-symmetric part is  
 not continuous.

For any fixed $\mu \in \mathbb{R} \setminus \{0\}$, we define
\begin{equation}
D(x, y) = \left(
\begin{array}{cc}
0 & d(x, y) \\
-d(x, y) & 0
\end{array}
\right)  \qquad 
\textrm{with } \quad d(x, y) = C_\mu
\begin{cases}
 \arctan \left( \frac{y}{x} \right) & x \neq 0, \, \\
  \frac{\pi}{2} \sgn(y) & x = 0 \,,
 \end{cases}
\end{equation}
where $C_\mu = \frac{ \mu^2 - 1}{\mu} \neq 0$ for $\mu \neq \pm 1$.  Moreover, let
\[
\A(x) = \mathbf{I}_{2} + D(x) = \left(
\begin{array}{cc}
1 & d(x, y) \\
-d(x, y) & 1
\end{array}
\right) .
\]
Note that $D$ is skew-symmetric and the matrix $\A$ is uniformly elliptic as   
\begin{equation}
\wei{\A(x, y) \xi, \xi}  = |\xi|^2, \quad \forall \ \xi \in \mathbb{R}^2, \quad \text{for all a.e.} (x,y) \in \R^2.
\end{equation}
Furthermore, since arc-tangent is bounded smooth function, the matrix $\A$ has bounded coefficients which are smooth away from the $y$-axis: $Y = \{(0, y) : y \in \mathbb{R} \}$.
 Note also that $d$ cannot be extended as a continuous function to $Y$ for $\mu \neq \pm 1$, that is, $C_\mu \neq 0$  since 
\begin{equation}
  \lim_{x \to 0^+} d(x, y) =    \frac{C_\mu \pi}{2} \sgn(y) \neq -  \frac{C_\mu \pi}{2} \sgn(y) =  \lim_{x \to 0^-} d(x, y)\,.
\end{equation}
We also remark that $d$ is bounded with $\norm{d}_{L^\infty(\R^2)} = \frac{\pi|1-\mu^2|}{2\mu}$, and thus $d$ is a \textup{BMO}-function.

Let $B_1$ be the unit ball in $\mathbb{R}^2$ centered at the origin. We consider the following elliptic equation
\begin{equation}
Lu: = \textrm{div}(\A(x,y) \nabla u) =  0 \quad \text{in} \quad B_1.
\end{equation}
We show that for each $\mu \in (0,1)$  the function $u(x, y) = x (x^2 + y^2)^{\frac{\mu - 1}{2}}$ is a weak solution of $L u = 0$.  
First note that $u$ is smooth in $\mathbb{R}\setminus \{0\}$ and $d$ is smooth on $\mathbb{R}^2 \setminus Y$. Then on $\mathbb{R}^2 \setminus Y$
\begin{align}
 L u = \Delta u + (d u_y )_x - (du_x)_y = \Delta u + d_x u_y  -  d_y u_x \,.
\end{align}
Using the polar coordinates we have
\begin{align*}
Lu &= u_{rr} + \frac{1}{r} u_r + \frac{1}{r^2} u_{\theta \theta} + \left(d_r \cos \theta - d_{\theta} \frac{\sin \theta}{r} \right) \left(u_r \sin \theta + u_{\theta} \frac{\cos \theta}{r} \right) \\
&\qquad - 
\left(u_r \cos \theta - u_{\theta} \frac{\sin \theta}{r} \right) \left(d_r \sin \theta + d_{\theta} \frac{\cos \theta}{r} \right) \\
&=  u_{rr} + \frac{1}{r} u_r + \frac{1}{r^2} u_{\theta \theta} + \frac{ d_r u_\theta - d_\theta u_r}{r}\,.
\end{align*}
Moreover,  $u = r^\mu \cos \theta$ for any $r$ and $\theta$.   Also for $r > 0$ and $\theta \in (-\pi/2, \pi/2)$
\begin{equation*}
d = C_\mu \arctan \left(\tan \theta\right) =  C_\mu \theta   \,,
\end{equation*}
and for $\theta \in (\pi/2, 3\pi/2)$ 
\begin{equation*}
d = C_\mu \arctan \left(\tan \theta\right) =  C_\mu (\theta - \pi)  \,.
\end{equation*}
In particular, $d_r  = 0$ and $d_\theta = C_\mu$.  Hence, in $\mathbb{R}^2 \setminus Y$ the definition of $C_\mu$ yields
\begin{align} \label{aei}
\frac{1}{r^{\mu - 2}} Lu =  \mu (\mu - 1) \cos \theta + \mu \cos \theta - \cos \theta - C_\mu \mu \cos \theta = 0 \,.
\end{align}
Hence, we proved that $Lu = 0$ on $\R^2 \setminus Y$, where $Y = \{(0, y): y \in \R\}$ is the $y$-axis.  

Let us prove that for $\mu \in (0,1)$, $u$ is a weak solution of $Lu = 0$. Observe that for such $\mu$ one has $|\nabla u| \in L^2(B_1)$. Let $\varphi$ be any smooth function compactly supported in $B_1$. We claim that 
\begin{equation} \label{sol-weak}
\int_{B_1} \wei{\A(x, y) \nabla u, \nabla \varphi} dxdy = 0 \,.
\end{equation} 
Let $\epsilon \in (0,1)$, and we write
\begin{equation} \label{epsilon-ball}
\int_{B_1} \wei{\A(x, y) \nabla u, \nabla \varphi} dxdy = \int_{B_\epsilon}\wei{\A(x, y) \nabla u, \nabla \varphi} dxdy + \int_{B_1\setminus B_\epsilon}\wei{\A(x, y) \nabla u, \nabla \varphi} dxdy 
\end{equation}
We now control the first term on the right hand side of \eqref{epsilon-ball}. As $|\nabla u(x,y)| \leq  C(\mu) r^{\mu-1}$, we infer that
\begin{equation} \label{term-1.est}
\begin{split}
\left|  \int_{B_\epsilon}\wei{\A(x, y) \nabla u, \nabla \varphi} dxdy \right| & \leq \norm{\varphi}_{L^\infty(B_1)} \norm{\A}_{L^\infty(B_1)} \int_{B_\epsilon} |\nabla u(x,y)| dxdy \\
& \leq C \int_0^\epsilon r^{\mu} dr \leq  C \epsilon^{\mu+1} \rightarrow 0 \quad \text{as} \quad \epsilon \rightarrow 0^+.
\end{split}
\end{equation}
To deal with the second term on the right hand side of \eqref{epsilon-ball}, let us denote 
\[ 
\begin{split}
& E^+ = \{(x, y) : x > 0\} \cap (B_1\setminus B_\epsilon), \quad  E^- = \{(x, y) : x < 0\} \cap (B_1\setminus B_\epsilon), \\
& Y_\epsilon = Y \cap\{(x,y) \in \mathbb{R}^2:\epsilon < |y| < 1\},
\end{split}
\] 
and we write
\[
\int_{B_1\setminus B_\epsilon} \wei{\A(x, y) \nabla u, \nabla \varphi} dxdy
= \int_{E^+} \wei{\A(x, y) \nabla u, \nabla \varphi} dxdy
+ \int_{E^-} \wei{\A(x, y) \nabla u, \nabla \varphi} dxdy.
\]
Then, by using the integration by parts on each of $E^+$ and $E^-$, and using $Lu = 0$ on $\R^2 \setminus Y$, we obtain
\begin{align}
& \int_{B_1 \setminus B_\epsilon} \wei{\A(x, y) \nabla u, \nabla \varphi} dxdy \nonumber \\
&=  \int_{Y_\epsilon} \wei{\A(0^+, y) \nabla u(0^+, y), -e_1} \varphi(0, y) dy 
+ \int_{Y_\epsilon} \wei{\A(0^-, y) \nabla u(0^-, y), e_1 }\varphi(0,y) dy \nonumber \\
&\quad \quad- \int_{\partial B_\epsilon} \wei{\A \nabla u, \vec{\nu}} \varphi(x,y) dS - \int_{E^+} Lu(x,y) \varphi(x,y) dxdy 
- \int_{E^-} Lu(x,y) \varphi(x,y) dxdy  \nonumber\\
&= -\int_{Y_\epsilon}[ u_x(0^+, y) + d(0^+, y)u_y(0^+, y)] \varphi (0,y) dy  \nonumber \\ \label{first-cal-sol}
& \quad \quad
+ \int_{Y_\epsilon}  [u_x(0^-, y) + d(0^-, y)u_y(0^-, y) ] \varphi (0, y) dy - \int_{\partial B_\epsilon} \wei{\A \nabla u, \vec{\nu}} \varphi(x,y) dS
\,,
\end{align}
where $e_1 = (1, 0)^T$, $\vec{\nu}$ is the normal outward vector on the circle $\partial B_\epsilon$, and for any function $f$ we denote $f(0^\pm, y) = \lim_{x \to 0^\pm} f(x, y)$. Next, since $u_x$ is a continuous function on 
$\R^2 \setminus \{0\}$, we obtain 
\begin{equation} \label{compatibility-1}
-\int_{Y_\epsilon} u_x(0^+, y) \varphi(0, y) dy 
+ \int_{Y_\epsilon}  u_x(0^-, y)  \varphi(0, y)dy = 0 \,.
\end{equation}
Moreover, since $u_y(0^+, y) = u_y(0^-, y) =0$ for each $y \in Y_\epsilon$ and $d$ is bounded, it follows that  
\begin{equation} \label{compatibility-2}
-\int_{Y_\epsilon}  d(0^+, y)u_y(0^+, y) \varphi(0, y) dy 
+ \int_{Y_\epsilon}  d(0^-, y)u_y(0^-, y)  \varphi(0, y) dy = 0.
\end{equation}
Finally,  $|\nabla u| \leq C r^{\mu-1}$, and boundedness of $\A$ imply
\begin{equation} \label{sphere-ep}
\begin{split}
\left| \int_{\partial B_\epsilon} \wei{\A \nabla u, \vec{\nu}} \varphi(x,y) dS \right|  &\leq \norm{\A}_{L^\infty(B_1)} \norm{\nabla \varphi}_{L^\infty(B_1)}\epsilon^{\mu-1} \int_{\partial B_\epsilon} dS\\
& \leq C \epsilon^{\mu} \rightarrow 0, \quad \text{as} \quad \epsilon \rightarrow 0^+.
\end{split} 
\end{equation}
Hence, by collecting \eqref{epsilon-ball}, \eqref{term-1.est}, \eqref{first-cal-sol}, \eqref{compatibility-1}, \eqref{compatibility-2}, and \eqref{sphere-ep}, we see that \eqref{sol-weak} follows as desired.

Finally, it can be easily inferred that $|\nabla u| \in L^p(B_1)$   if and only if 	$p < \frac{2}{1 - \mu}$
and $u \in C^\alpha(B_1)$ if and only if $\alpha \leq \mu$.  Since $\mu \in (0, 1)$ was arbitrary, the desired conclusion follows.

 \noindent
\textbf{Acknowledgement.}  T. Phan's research is supported by the Simons Foundation, grant \#~354889. The authors would like to thank Hongjie Dong for finding 
 a gap in the proof of Theorem \ref{exam-thrm} and providing helpful references. 


\begin{thebibliography}{10}

\bibitem{Ben-Fre} A. Bensoussan, J. Frehse, {\it  Regularity results for nonlinear elliptic systems and applications}, Applied Math. Sciences, Vol 151, Springer, 2002.
 

\bibitem{CLMS} R. Coifman, P.-L. Lions, Y. Meyer, and S. Semmes, {\it Compensated compactness and Hardy spaces}. J. Math. Pures Appl. (9) 72 (1993), no. 3, 247 - 286.

\bibitem{Dong-Kim} H. Dong, S. Kim, {\it  Fundamental solutions for second order parabolic systems with drift terms},  arXiv:1707.09162.


\bibitem{FS} C. Fefferman, E. Stein, {\it $\mathcal{H}^p$ spaces of several variables}.  Acta Math. 129 (1972), 137 - 193.

\bibitem{Giaquinta} M. Giaquinta, {\it Multiple integrals in the calculus of variations and nonlinear elliptic systems}, Princeton University Press, 1983.

\bibitem{Gia-Modica} M. Giaquinta, G. Modica, {\it Regularity results for some classes of higher order non linear elliptic systems}. J. Reine Angew. Math. 311/312, 145 -169 (1979).

\bibitem{Gia-Stru} M. Giaquinta, M. Stuwe, {\it On the partial regularity of weak solutions of nonlinear parabolic systems}, Math. Z. 179 (1982) 437 - 451.

\bibitem{Giusti} E. Giusti, {\it Direct Methods in the Calculus of Variations}, World Scientific Publishing Company, Tuck Link (2003).

\bibitem{KS} V. F. Kovalenko, Yu. A Semenov,  {\it  Co-semigroups in the spaces $L^p(\mathbb{R}^d)$ and $\hat{C}(\mathbb{R}^d)$ generated by $b \cdot \nabla $} (Russian) Teor. Veroyatnost. i Primenen. 35(3), 449-458(1990); Translation in Theory Probab. Appl. 35(3), 443-453 (1990).

\bibitem{LP} L. Li, J. Pipher, {\it Boundary behavior of solutions of elliptic operators in divergence form with a BMO anti-symmetric part},  arXiv:1712.06705

\bibitem{LS} V. Liskevich, Y. Semenov, {\it Estimates for fundamental solutions of second-order parabolic equations.} J. Lond. Math. Soc. (2) 62(2), 521-543 (2000).

\bibitem{MV} V. G. Mazya, I. E. Verbitsky, {\it Form boundedness of the general second-order differential operator} Comm. Pure Appl. Math. 59 (2006), no. 9, 1286 - 1329. 


\bibitem{M} N.G.~Meyers. {\it An $L^p$-estimate for the gradient of solutions of second order elliptic divergence equations.} Ann. Sc. Norm. Super. Pisa Cl. Sci. (3) 17 (1963) 189â 189-206.

\bibitem{O} H. Osada, {\it Diffusion processes with generators of generalized divergence form.} J. Math. Kyoto Univ 27(4), 597-619 (1987).

\bibitem{TP} T. Phan, {\it Regularity gradient estimates for weak solutions of singular quasi-linear parabolic equations}. J. Differential Equations 263 (2017), no. 12, 8329-8361.

\bibitem{TP1} T. Phan, {\it Lorentz estimates for weak solutions of quasi-linear parabolic equations with singular divergence-free drifts}, Canadian Journal of Mathematics, DOI:10.4153/CJM-2017-049-3.

\bibitem{Sverak} G. Seregin, L. Silvestre,  V. Sverak, A. Zlatos,  {\it On divergence-free drifts}. J. Differential Equations 252 (2012), no. 1, 505--540.

\bibitem{Q-Zhang} Q. S. Zhang, {\it A strong regularity result for parabolic equations}, Comm. Math. Phys. 244 (2004) 245-260.
\end{thebibliography}
\end{document}